\documentclass[]{amsart}

\usepackage{amsmath}
\usepackage{amsthm}
\usepackage{amsfonts}
\usepackage{amssymb}
\usepackage[shortlabels]{enumitem}
\usepackage{todonotes}
\usepackage{mathtools}
\usepackage{csquotes}
\usepackage{changes}

\usepackage[style=alphabetic,sorting=nyt,backend=bibtex8, maxbibnames=99]{biblatex}
\bibliography{reference.bib}

\usepackage{hyperref}
\hypersetup{colorlinks=true}

\usepackage{graphicx}

\newcommand{\upupharpoonright}{%
  \upharpoonright
  \mathrel{\mkern-4.5mu}%
  \upharpoonright
}

\newcommand{\ubar}[1]{\text{\b{$#1$}}}

\makeatletter
\newcommand{\oast}{\mathbin{\mathpalette\make@circled\ast}}
\newcommand{\make@circled}[2]{%
  \ooalign{$\m@th#1\smallbigcirc{#1}$\cr\hidewidth$\m@th#1#2$\hidewidth\cr}%
}
\newcommand{\smallbigcirc}[1]{%
  \vcenter{\hbox{\scalebox{0.77778}{$\m@th#1\bigcirc$}}}%
}
\makeatother

\newtheorem{theorem}{Theorem}[]
\newtheorem{proposition}[theorem]{Proposition}
\newtheorem{corollary}[theorem]{Corollary}
\newtheorem{lemma}[theorem]{Lemma}
\newtheorem{claim}{Claim}[theorem]
\newtheorem{subclaim}{Claim}[claim]
\theoremstyle{definition}

\newtheorem{question}{Question}
\newtheorem*{question*}{Question}

\theoremstyle{remark}

\begin{document}

\title[Real degrees in the side-by-side Sacks model]{Constructibility real degrees in the\\side-by-side Sacks model}

\author{Lorenzo Notaro}
\address{University of Vienna, Institute of Mathematics, Kurt G\"{o}del Research Center, Kolingasse 14-16, 1090 Vienna, Austria}
\curraddr{}
\email{lorenzo.notaro@univie.ac.at}
\date{}

\begin{abstract}
We study the join-semilattice of constructibility real degrees in the side-by-side Sacks model, i.e., the model of set theory obtained by forcing with a countable-support product of infinitely many Sacks forcings over the constructible universe. In particular, we prove that in the side-by-side Sacks model the join-semilattice of constructibility real degrees is rigid, i.e., it does not have non-trivial automorphisms.
\end{abstract}

\thanks{The author would like to acknowledge INdAM for the financial support. This research was also partially supported by the project PRIN 2022 “Models, sets and classifications”, prot. 2022TECZJA. This research was funded in whole or in part by the Austrian Science Fund (FWF) \href{https://www.fwf.ac.at/en/research-radar/10.55776/ESP1829225}{10.55776/ESP1829225}. For open access purposes, the author has applied a CC BY public copyright license to any author accepted manuscript version arising from this submission.}

\subjclass[2020]{Primary 03E35, Secondary 03E45}
\keywords{constructibility degrees, real degrees, Sacks forcing, side-by-side Sacks model, Sacks model}
\maketitle

\section{Introduction}

In \cite{MR276079}, Sacks introduced the perfect-set forcing---i.e., the forcing notion consisting of perfect closed sets of reals ordered by inclusion. This forcing, also known as \emph{Sacks forcing}, has been widely used in descriptive set theory due to its feature of adding a particularly tame generic real of minimal degree of constructibility. 

In \cite{MR556894}, Laver and Baumgartner introduced the \emph{iterated Sacks model}. This model is obtained by forcing over a model of $\mathsf{CH}$ (often the constructible universe) with a countable-support iteration of $\omega_2$-many Sacks forcings. It has been the subject of intensive study (see e.g. \cites{MR3981132,MR1942314, MR2640071, MR1978947}), mainly due to its rich combinatorial theory, well enucleated by Pawlikowski and Ciesielski's \emph{Covering Property Axiom} ($\mathsf{CPA}$) \cite{MR2176267}.

Furthermore, the \emph{side-by-side Sacks model}, which is obtained by forcing over a model of $\mathsf{CH}$ with a countable-support product of infinitely many Sacks forcings, has also been studied (see e.g. \cites{MR794485, MR987317, MR1900391, MR3702751}). 

Much is already known on the structure of the constructibility real degrees in models obtained by forcing over the constructible universe with either an iteration or a finite product of Sacks forcings: forcing with a countable-support iteration of $\omega_2$-many Sacks forcings results in the constructibility real degrees being well-ordered with order-type $\omega_2$ \cites{MR556894, MR2940963}; forcing with a product of $n$ Sacks forcings, for some $n\in\omega$, results in the constructibility real degrees being isomorphic to the powerset lattice of $n$ (see e.g. \cite{MR1777770}).

Here we address the case that has been less explored: What properties do the constructibility real degrees satisfy in the side-by-side Sacks model? 

Our main results are the following, which show that an infinite product of Sacks forcings behaves very differently, at least real-degrees-wise, compared to a finite product. Here, $\mathrm{L}[G]$ is the generic extension of $\mathrm{L}$ obtained by a countable-support product of infinitely many Sacks forcings.

\begin{theorem}\label{thm:main1}
In $\mathrm{L}[G]$, $(\mathcal{D}_c, \le_c)$ is neither a meet-semilattice, nor $\sigma$-complete, nor complemented.
\end{theorem}

\begin{theorem}\label{thm:main2}
In $\mathrm{L}[G]$, $(\mathcal{D}_c, \le_c)$ is rigid, i.e., it has no non-trivial automorphisms.
\end{theorem}

\begin{theorem}\label{thm:main3}
In $\mathrm{L}[G]$, apart from the least and greatest (if it exists) real degrees, no other real degree is definable in $(\mathcal{D}_c, \le_c)$.
\end{theorem}

In Section~\ref{sec:pre}, we briefly discuss some basic definitions regarding Sacks forcing and its products. In Section~\ref{sec:representation}, we prove a representation theorem for the join-semilattice $(\mathcal{D}_c, \le_c)$ in $\mathrm{L}[G]$ (Theorem~\ref{thm:representation}). This representation is key for the proofs of Theorems~\ref{thm:main1}-\ref{thm:main3}, which are presented in Section~\ref{sec:mainproofs}. Finally, in Section~\ref{sec:last}, we prove the following result in $\mathsf{ZF}$ showing that, in some sense, we cannot improve our representation theorem. 

\begin{theorem}\label{thm:main4}
$(\mathcal{P}(\omega), \subseteq)$ is not isomorphic to any ideal of $(\mathcal{D}_c, \le_c)$. 
\end{theorem}

\section{Preliminaries}\label{sec:pre}

Our notation is standard---see e.g.~\cite{MR1940513}. Following set-theoretic practice, we identify $\mathcal{P}(\omega)$ with the Cantor space ${}^\omega 2$, and refer to their members as ``reals". A \emph{tree} $T$ is a nonempty set of finite sequences closed under initial segments, and we denote by $[T]$ the set of its infinite branches. Given two sets $x,y$, we denote by ${}^x y$ the set of all functions from $x$ to $y$.

\subsection{Constructibility real degrees}\label{sec:pre:realdegrees}

This subsection is taken for the most part from \cite[\S 2.2]{Sierpinski}. The \emph{constructibility preorder} $ \le_c $ on the universe of sets is defined by 
$$ 
 x \le_c y \iff x \in \mathrm{L} (y).
$$
Let $ \equiv_c $ be the induced equivalence relation. We are interested mainly in the restriction of the contructibility preorder $\le_c$ on $\mathcal{P}(\omega)$. The quotient $ \mathcal{P} ( \omega ) / {\equiv_c} $ with the induced order is the set of \emph{constructibility real degrees}, and it is denoted by $ \mathcal{D}_c $. From this point onward, we will refer to the constructibility real degrees simply as \emph{real degrees}. Given a real $a$, we will denote by its bold character $\mathbf{a}$ its equivalence class $[a]_c$.
If $ \mathbf{d} \in \mathcal{D}_c $ then set $ \mathrm{L} [ \mathbf{d} ] = \mathrm{L} [ a ] $ for some/any $ a \in \mathbf{d} $.
Note that $ \mathcal{D}_c $ is a partial order with minimum $ \mathbf{0} = \mathcal{P}(\omega)\cap \mathrm{L} $.  

The poset $(\mathcal{D}_c, \le_c)$ is actually a join-semilattice with $ \mathbf{c} \vee \mathbf{d} =  c \oplus d  $ for some/any $ c \in \mathbf{c} $ and $ d \in \mathbf{d} $, where 
$$ 
c \oplus d = \{ 2 n \mid n \in c \} \cup \{ 2 n + 1 \mid n \in d \} . 
$$
More generally, if $ \langle \cdot , \cdot \rangle \colon \omega \times \omega \to \omega$ is a recursive bijection, and $ a_n \subseteq \omega $, then letting
\[
\textstyle\bigoplus_{ n \in \omega } a_n \coloneqq \{ \langle n , k \rangle \mid n \in \omega \text{ and } k \in a_n \}
\]
we have that $ a_i \le_c \bigoplus_{ n \in \omega } a_n $ for all \( i < \omega \). Thus, $\mathcal{D}_c$ is \( \sigma \)-directed---i.e., any countable set of real degrees has an upper bound.

On the other hand, a countable set of real degrees need not have a least upper bound, i.e., $ \mathcal{D}_c $ need not be a $\sigma$-complete~\cite{MR520202}. Again, in~\cite{MR520202}, it is also shown that $\mathcal{D}_c$ needs not be a meet-semilattice.

The structure of $ \mathcal{D}_c $ is highly non-absolute. If $ \mathrm{V} = \mathrm{L} $ then $ \mathcal{D}_c $ is the singleton $ \{ \mathbf{0} \} $, while if $ \mathrm{V} = \mathrm{L} [ c ] $ where $ c $ is a Cohen real, then $ \mathcal{D}_c $ has the size of the continuum and a rich structure~\cite{MR850218}. 

Adamowicz~\cite{MR505487} has shown that for every constructible, constructibly countable, and well-founded join-semilattice with a least element, there is a generic extension of the constructible universe in which $\mathcal{D}_c$ is isomorphic to the given join-semilattice---see~\cites{MR968536, MR1233817} for stronger results and more discussion on this. 

Here are a couple of other interesting results in this regard: as already mentioned in the introduction, in the iterated Sacks model $\mathcal{D}_c$  is well-ordered of order-type $\omega _2$; Groszek~\cite{MR1295981} has shown that $\mathcal{D}_c$ can, consistently with $\mathsf{ZFC}$, be isomorphic to the reverse copy of $\omega_1+1$; see also~\cites{MR1777770, MR1680978} for similar results.

\subsection{Sacks forcing and its products}\label{sec:pre:sacks}

We mostly adhere to the notation used in~\cite{MR2155534}. A tree $T\subseteq {}^{<\omega}2$ is a \emph{perfect binary tree} if every node of $T$ has two incomparable extensions in $T$. The poset of all perfect binary trees ordered by inclusion is known as \emph{Sacks forcing}, and it is denoted by $\mathbb{S}$. Clearly, $\mathbf{1}_\mathbb{S} = {}^{<\omega} 2$. The \emph{stem} of a condition $p \in \mathbb{S}$ is the $\subseteq$-maximal node $t \in p$ such that, for every $s \in p$, either $s \subseteq t$ or $t \subseteq s$.

Given $p \in \mathbb{S}$ and $n\in\omega$ we let $p^n$ be the set of all those $t \in p$ that are $\subseteq$-minimal in $p$ with respect to the property of having exactly $n$ proper initial segments that have two immediate successors in the tree $p$. For $p,q \in \mathbb{S}$ and $n\in\omega$, we let $$p \le_n q \iff p \le q \text{ and } p^n = q^n.$$

If $p \in \mathbb{S}$ and $n \in \omega$, there is a natural way of assigning to each finite binary sequence $\sigma \in {}^n 2$ an element $p(\sigma)$ of $p^n$. Let $p \ast \sigma \coloneqq \{s \in p \mid s \subseteq p(\sigma) \vee p(\sigma) \subseteq s\}$.

Given a cardinal $\kappa$, we denote by $\mathbb{S}^\kappa$ the countable-support product of $\kappa$-many Sacks forcing. A condition $p$ of $\mathbb{S}^\kappa$ is a map from $\kappa$ to $\mathbb{S}$ such that $\text{supp}(p) \coloneqq \{\alpha \in \kappa \mid p(\alpha) \neq \mathbf{1}_{\mathbb{S}}\}$, known as the \emph{support} of $p$, is countable. For any subset $D\subseteq \kappa$, we let $\mathbb{S}^\kappa \upharpoonright D$ denote the complete subforcing of $\mathbb{S}^\kappa$ defined as the set of all the conditions of $\mathbb{S}^\kappa$ whose support is included in $D$. Note that $\mathbb{S}^\kappa \upharpoonright D$ is isomorphic to $\mathbb{S}^{|D|}$. Given a condition $p \in \mathbb{S}^\kappa$, we abuse the notation and denote by $p \upharpoonright D$ the condition of $\mathbb{S}^\kappa \upharpoonright D$ defined in the expected way: $(p \upharpoonright D)(\alpha) = p(\alpha)$ if $\alpha \in D$, and $(p \upharpoonright D)(\alpha) = \mathbf{1}_\mathbb{S}$ otherwise.

Given $p, q \in \mathbb{S}^\kappa$, some finite $F\subseteq \kappa$ and some $n\in\omega$, we let $$p \le_{F,n} q \iff p \le q \text{ and } \forall \alpha \in F \ \big(p(\alpha) \le_n q(\alpha)\big).$$ A \emph{fusion sequence} is a sequence $(p_n)_{n\in\omega}$ of elements of $\mathbb{S}^\kappa$ such that there exists a $\subseteq$-increasing sequence $(F_n)_{n \in\omega}$ of finite sets with
\begin{enumerate}
\itemsep0.3em
\item $p_{n+1} \le_{F_{n}, n} p_{n}$ for every $n \in\omega $, and
\item $\bigcup_{n \in\omega} F_n = \bigcup_{n\in\omega} \text{supp}(p_n)$.
\end{enumerate}

For every fusion sequence $(p_n)_{n\in\omega}$, we let 
$$
\mathcal{F}((p_n)_{n\in\omega}) \coloneqq \Big( \bigcap_{n\in\omega} p_n(\alpha)\Big)_{\alpha \in \kappa}
$$
be its \emph{fusion}. It is easy to check that a fusion is always an element of $\mathbb{S}^\kappa$.

If $p \in \mathbb{S}^\kappa$, $F \subseteq \kappa$ and $n \in \omega$, and $\sigma \in {}^F({}^n 2)$, let $p \ast \sigma$ be such that for all $\alpha \in F$, $(p\ast \sigma)(\alpha) = p(\alpha) \ast \sigma(\alpha)$ and for all $\alpha \not\in F$, $(p\ast \sigma)(\alpha) = p(\alpha)$.  

If $\kappa = \omega$, we write $p \le_n q$ instead of $p \le_{n,n} q$, and a fusion sequence is simply a sequence $(p_n)_{n\in\omega}$ of elements of $\mathbb{S}^\omega$ such that $p_{n+1} \le_n p_{n}$ for every $n \in\omega$.

The forcing $\mathbb{S}^\kappa$ is \emph{proper} and, in particular, does not collapse $\aleph_1$ \cite{MR2155534}. Moreover, if $\mathsf{CH}$ holds, a simple $\Delta$-system argument shows that $\mathbb{S}^\kappa$ has the $\aleph_2$-cc, and therefore does not collapse any cardinal.

\section{Representing the real degrees in the side-by-side Sacks model}\label{sec:representation}

Let $\kappa$ be an infinite cardinal and fix an $\mathbb{S}^\kappa$-generic filter $G$ over $\mathrm{L}$. Let $\langle s_\alpha \mid \alpha \in \kappa\rangle$ be the generic sequence of Sacks reals added by $G$, i.e., for each $\alpha \in \kappa$,  $s_\alpha$ is the unique element of $\bigcap_{p \in G}[p(\alpha)]$. Let $S \coloneqq \{s_\alpha \mid \alpha \in \kappa\}$ be the set of these reals. 

In $\mathrm{L}[G]$, define the set $\mathcal{R}$ as follows:
\[
\mathcal{R} \coloneqq \big\{x \in [\kappa]^{\le\omega} \mid \forall \alpha \in \kappa \ (s_\alpha \le_c x \Rightarrow \alpha \in x)\big\}.
\]

This section is devoted to the proof of the following theorem, which is key to proving Theorems~\ref{thm:main1}-\ref{thm:main3}, as it unravels much of the combinatorics of the real degrees in $\mathrm{L}[G]$.

\begin{theorem}\label{thm:representation}
In $\mathrm{L}[G]$, $(\mathcal{D}_c, \le_c) \cong (\mathcal{R}, \subseteq)$.
\end{theorem}

Specifically, we will show in the proof of Theorem~\ref{thm:representation} that in $\mathrm{L}[G]$, the map $\Omega$ which associates to each real degree $\mathbf{r}$ the set $\{\alpha \in \kappa \mid \mathbf{s}_\alpha \le_c \mathbf{r}\}$ is an isomorphism from $(\mathcal{D}_c, \le_c)$ to $(\mathcal{R}, \subseteq)$. The fact that $\Omega$ is an embedding from $(\mathcal{D}_c, \le_c)$ to $[\kappa]^{\le \omega}$ is already known (see e.g. \cite[Theorem 1(2)]{MR1777770}). The point of Theorem~\ref{thm:representation} is to give a non-trivial characterization of the range of $\Omega$. It is precisely this characterization (i.e., the introduction of $\mathcal{R}$) which is crucial for proving Theorems~\ref{thm:main1}-\ref{thm:main3}.

Before carrying on with the proof, let us highlight that the choice of the constructible universe as our ground model is not due to its particular properties, which we do not employ, but instead to the fact that we are interested in studying the \emph{constructibility degrees} of the generic extension. Indeed, our results would also hold if we were to choose a different ground model $\mathrm{V}$, but then we would need to talk about $\mathrm{V}$-degrees rather than constructibility degrees.

Given a constructible $D\subseteq \kappa$, we let $G \upharpoonright D$ be the set of all the conditions in $G$ whose support is contained in $D$. We denote by $\dot{G} \upharpoonright D$ its canonical name. Note that $G \upharpoonright D$ is an $\mathbb{S}^\kappa \upharpoonright D$-generic filter over $\mathrm{L}$.

In order to prove Theorem~\ref{thm:representation}, we first need some  preliminary technical results. The first one tells us that we can often assume $\kappa=\omega$ without loss of generality. This assumption simplifies the construction of fusion sequences, at least notationally. 

A constructible set is said to be \emph{constructibly countable} if it is countable in $\mathrm{L}$.

\begin{lemma}\label{lemma:general} In $\mathrm{L}[G]$, for every $E \in [\kappa]^{\le \omega}$ there exists a constructible, constructibly countable $D\subseteq \kappa$ such that $E\subseteq D$ and $E \in \mathrm{L}[G \upharpoonright D]$. 
\end{lemma}
\begin{proof}
We work in $\mathrm{L}$. Fix some $p \in \mathbb{S}^\kappa$, some $\mathbb{S}^\kappa$-name $\dot{E}$ for $E$ and a $\mathbb{S}^\kappa$-name $\dot{f}$ such that 
\[
p \Vdash \dot{f}:\omega\rightarrow \kappa \text{ with } \mathrm{ran}(\dot{f}) = \dot{E}.
\]

Via a simple bookkeeping argument, we can inductively define a sequence $(p_n, F_n)_{n\in\omega}$ and a family of ordinals $\langle \alpha_\sigma \mid \sigma \in {}^{F_n}({}^n 2) \text{ for some }n \in\omega\rangle$ such that 
\begin{enumerate}[label={\roman*.}]
\itemsep0.3em
\item $(p_n)_{n\in\omega}$ is a fusion sequence witnessed by $(F_n)_{n \in\omega}$,
\item $p_0 = p$,
\item for all $n \in\omega$ and for all $\sigma \in {}^{F_n}({}^n 2)$, $p_{n+1} \ast \sigma \Vdash \dot{f}(n) = \alpha_\sigma$,
\item for all $n \in\omega $ and for all $\sigma \in {}^{F_n}({}^n 2)$, $\alpha_\sigma \in F_{n+1}$.
\end{enumerate}

Let $q$ be the fusion of the $p_n$s and let $D$ be its support. Then it follows from our construction that $q$ forces $\dot{E}\subseteq D \text{ and }\dot{E} \in \mathrm{L}[\dot{G} \upharpoonright D]$. By density, we are done.
\end{proof}
Note that for every infinite countable $D\subseteq \kappa$, $\mathbb{S}^\kappa \upharpoonright D \cong \mathbb{S}^\omega$. In particular, Lemma~\ref{lemma:general} implies that every real added by $\mathbb{S}^\kappa$ belongs to some $\mathbb{S}^\omega$-generic extension.

The next proposition tells us that any countable subset of $S$ can construct its own enumeration induced by the generic filter $G$. 

\begin{proposition}\label{prop:definA}
In $\mathrm{L}[G]$, for every $A \in [S]^{\le\omega}$, if we let $e_A:A\rightarrow \kappa$ be defined by $e(s_\alpha) = \alpha$ for every $s_\alpha \in A$, then:
\begin{enumerate}[label={\upshape \arabic*)}]
\itemsep0.3em
\item\label{prop:definA-1} $e_A \le_c A$.
\item\label{prop:definA-2} $\mathrm{L}(A) \vDash ``A$ is countable''.
\end{enumerate}
\end{proposition}

\begin{proof}
By Lemma~\ref{lemma:general}, there exists a constructible, constructibly countable $D\subseteq \kappa$ such that $\mathrm{ran}(e_A) \subseteq D$ and  $\mathrm{ran}(e_A) \in \mathrm{L}[G \upharpoonright D]$. Since $D$ is constructibly countable, \ref{prop:definA-2} directly follows once we prove \ref{prop:definA-1}. Moreover, as both $A$ and $e_A$ belong to $\mathrm{L}[G \upharpoonright D]$, we can suppose without loss of generality $\kappa = \omega$ (see the remark after Lemma~\ref{lemma:general}).

We now show that there must exist a $q \in G$ such that $[q(n)] \cap [q(m)] = \emptyset$ for every distinct $n,m\in\omega$. We work in $\mathrm{L}$. Fix any $p \in \mathbb{S}^\omega$. It is routine to inductively define a fusion sequence $(p_n)_{n\in\omega}$ below $p$ such that for every $n\in\omega$, for every distinct $k, m < n$, $[p_{n+1}(k)] \cap [p_{n+1}(m)] = \emptyset$. Let $q$ be its fusion so that $q$ extends $p$ and satisfies the wanted property. By density, we can find such a $q$ in $G$.

By the properties of the condition $q$ and the fact that $q \in G$, we have that, in $\mathrm{L}[G]$,  $e_A(s)$ is the unique $n \in\omega$ such that $s \in [q(n)]$, for every $s \in A$. Since this definition is absolute modulo the parameters $A$ and $q$, and since $q$ is constructible, we conclude that $e_A$ is constructible relative to $A$, i.e., $e_A \le_c A$.
\end{proof}
\begin{corollary}\label{cor:realset}
In $\mathrm{L}[G]$, for every $A \in [S]^{\le\omega}$, there is a real $r$ such that $r \equiv_c A$.
\end{corollary}
\begin{proof}
If $A$ is finite, then the claim is trivial. So we can assume that $A$ is infinite. By Proposition \ref{prop:definA}, the set $A$ is countable in $\mathrm{L}(A)$. Thus, we can fix a surjection $\psi: \omega\rightarrow A$ in $\mathrm{L}(A)$. Let $r \coloneqq \bigoplus_{k\in\omega} \psi(k)$. Clearly, $A \le_c r$. Furthermore, since $\psi$ belongs to $\mathrm{L}(A)$, we also have $r \le_c A$.
\end{proof}

Given a $\mathbb{S}^\kappa$-name $\dot{r}$ for a real, and some condition $p \in \mathbb{S}^\kappa$, we let $\dot{r}_p$ be the longest initial segment of $\dot{r}$ decided by $p$. Note that if $p$ does not force $\dot{r}$ to belong to the ground model, then $\dot{r}_p$ is a finite sequence. For each $\alpha \in \kappa$, $\dot{s}_\alpha$ is the canonical $\mathbb{S}^\kappa$-name for the $\alpha$-th generic Sacks real.

\begin{proposition}\label{prop:realconstr}
Let $\alpha \in \kappa$ and $p\in\mathbb{S}^\kappa$ and let $\dot{r}$ be a $\mathbb{S}^\kappa$-name for a real. The following are equivalent:
\begin{enumerate}[label={\upshape \arabic*)}]
\itemsep0.3em
\item\label{prop:realconstr-1}$p \Vdash \dot{s}_\alpha \le_c \dot{r}$.
\item\label{prop:realconstr-2} $p \Vdash \dot{r} \not\in \mathrm{L}\big[\dot{G} \upharpoonright (\kappa\setminus\{\alpha\})\big]$.
\item\label{prop:realconstr-3} For every $q \le p$ there exist $q_0,q_1 \le q$ with $q_0 \upharpoonright (\kappa \setminus \{\alpha\}) = q_1 \upharpoonright (\kappa \setminus \{\alpha\})$ such that $\dot{r}_{q_0}$ and $\dot{r}_{q_1}$ are incomparable.
\end{enumerate}
\end{proposition}
\begin{proof}

$\text{\ref{prop:realconstr-1}}\Rightarrow \text{\ref{prop:realconstr-2}}$: By contraposition, suppose that $p$ does not force $\dot{r} \not\in \mathrm{L}[\dot{G} \upharpoonright (\kappa\setminus\{\alpha\})]$. Then there exists $q \le p$ such that $q$ forces $\dot{r} \in \mathrm{L}[\dot{G} \upharpoonright (\kappa\setminus\{\alpha\})]$. Then, since $\dot{s}_\alpha$ is always forced not to belong to $\mathrm{L}[\dot{G} \upharpoonright (\kappa\setminus\{\alpha\})]$ by mutual genericity, $q$ forces $\dot{s}_\alpha \not\le_c \dot{r}$. Therefore $p \not\Vdash \dot{s}_\alpha \le_c \dot{r}$.\vspace{0.3em}

$\text{\ref{prop:realconstr-2}} \Rightarrow \text{\ref{prop:realconstr-3}}$: Again by contraposition, suppose that there exists $q \le p$ such that for every $q_0,q_1 \le q$, if $q_0 \upharpoonright (\kappa \setminus \{\alpha\}) = q_1 \upharpoonright (\kappa \setminus \{\alpha\})$, then $\dot{r}_{q \ast \sigma_0}$ and $\dot{r}_{q \ast \sigma_1}$ are comparable. Equivalently, for any $z \le q$, any initial segment of $\dot{r}$ decided by $z$ is already decided by $z \upharpoonright (\kappa\setminus \{\alpha\})$. Thus, $q \Vdash \dot{r} \in \mathrm{L}[\dot{G} \upharpoonright (\kappa\setminus\{\alpha\})]$.\vspace{0.3em}

$\text{\ref{prop:realconstr-3}} \Rightarrow \text{\ref{prop:realconstr-1}}$: Suppose that $p$ and $\dot{r}$ satisfy the hypotheses of \ref{prop:realconstr-3}. By Lemma~\ref{lemma:general} and the remark afterward, we can suppose without loss of generality that $\kappa = \omega$. Hence, we will denote $\alpha$ by $n$ so to highlight that we are talking about a natural number. 

By density, it suffices to show that there exists a $q \le p$ such that $q \Vdash \dot{s}_n \le_c \dot{r}$. 

\begin{claim}\label{claim:1}
For every $m > n$, for every $q \le p$, for every $\sigma_0,\sigma_1 \in {}^{m}({}^{m} 2)$, if $\sigma_0(n) \neq \sigma_1(n)$, then there exists a $z \le_{m} q$ such that $\dot{r}_{z \ast \sigma_0}$ and $\dot{r}_{z\ast \sigma_1}$ are incomparable.
\end{claim}
\begin{proof}
Fix an $m > n$, a $q \le p$ and $\sigma_0,\sigma_1 \in {}^{m}({}^{m} 2)$ such that $\sigma_0(n) \neq \sigma_1(n)$. By hypothesis, there are $q_0,q_1 \le q\ast \sigma_0$ such that $q_0 \upharpoonright (\omega \setminus \{n\}) = q_1 \upharpoonright (\omega \setminus \{n\})$, and such that $\dot{r}_{q_0}$ and $\dot{r}_{q_1}$ are incomparable.

Let
\[
E \coloneqq \big\{k \in\omega \mid k \ge m \text{ or } ( k < m \text{ and } \sigma_0 (k) = \sigma_1(k))\big\}.
\] 
Now let $p' \le q\ast \sigma_1$ be defined as follows: for each $k$, if $k \in E$, then $p'(k) \coloneqq q_0(k)$; if $k \not\in E$, then $p'(k) \coloneqq (q\ast\sigma_1)(k) = q(k)\ast \sigma_1(k)$. 

Fix a $w \le p'$ such that $\dot{r}_w$ is incomparable with either $\dot{r}_{q_0}$ or $\dot{r}_{q_1}$.  Suppose without loss of generality that $\dot{r}_w$ is incomparable with $\dot{r}_{q_0}$ (otherwise substitute $q_0$ with $q_1$ in what follows). Then let $z \le q$ be defined as follows: for every $k \ge m$, $z(k) = w(k)$; for every $k < m$, let $z(k)^m = q(k)^m$ and, for every $\tau \in {}^{m}2$, 
\begin{equation}\label{eq:defz}
z(k) \ast \tau = \begin{cases}
q(k) \ast \tau \ &\text{ if } \tau \not\in \{\sigma_0(k),\sigma_1(k)\}\\
q_0(k) \ &\text{ if } k\not\in E \text{ and } \tau = \sigma_0(k)\\
w(k) \ &\text{ otherwise}
\end{cases}
\end{equation}

Let us check that $z$ is well-defined. More precisely, by fixing some $k < m$ and some $\tau \in {}^m 2$, we need to verify that $q(k)(\tau)$ is an initial segment of the stem of $z(k) \ast \tau$ as prescribed by \eqref{eq:defz}: if $\tau$ is different from both $\sigma_0(k)$ and $\sigma_1(k)$ there is nothing to show, as by \eqref{eq:defz} $z(k)\ast \tau = q(k)\ast \tau$; if $k \not\in E$ and $\tau = \sigma_0(k)$, then $q(k)(\sigma_0(k))$ is an initial segment of the stem of $q_0(k)$ because, by definition of $q_0$, $q_0(k) \le q(k) \ast \sigma_0(k)$; if $k \in E$ and $\tau = \sigma_0(k) = \sigma_1(k)$, then $q(k)(\sigma_0(k))$ is an initial segment of the stem of $w$ because, by definition of $w, p'$ and $q_0$, $w(k) \le p'(k)  = q_0(k) \le q(k) \ast \sigma_0(k)$; finally, if $k \not\in E$ and $\tau = \sigma_1(k)$, then $q(k)(\tau)$ is an initial segment of the stem of $w(k)$ because, by definition of $w$ and $p'$, $w(k) \le p'(k) = q(k)\ast \sigma_1(k)$.

Clearly, by definition of $z$, $z \le_m q$. Moreover, it directly follows from \eqref{eq:defz} that $z \ast \sigma_1 = w$.  Once we show $z \ast \sigma_0 \le q_0$ we are done, as it would imply that $\dot{r}_{q_0}$  is an initial segment of $\dot{r}_{z \ast \sigma_0}$, and this, together with the fact that $\dot{r}_{z \ast \sigma_1} = \dot{r}_w$, results in $\dot{r}_{z \ast \sigma_0}$ and $\dot{r}_{z \ast \sigma_1}$ being incomparable, by our choice of $w$. To see this, pick any $k \in \omega$: if $k \ge m$, then $(z \ast \sigma_0)(k) = z(k)$, and $z(k) = w(k)$ by definition of $z$, and $w(k) \le p'(k) = q_0(k)$ by the definition of $w$ and $p'$; if $k < m$ and $k \not\in E$, then $(z \ast \sigma_0)(k) = z(k) \ast \sigma_0(k) = q_0(k)$, by \eqref{eq:defz}; finally, if $k < m$ and $k \in E$, then $(z \ast \sigma_0)(k) = z(k) \ast \sigma_0(k) = w(k)$ by definition of $z$, but, since $k \in E$, $w(k) \le q_0(k)$.

Hence, $z \ast \sigma_0 \le q_0$, and we are done.
\end{proof}

\begin{claim}\label{claim:2}
There exists a $q \le p$ such that for every $m>n$, for every $\sigma_0,\sigma_1 \in {}^m({}^m 2)$, if $\sigma_0(n) \neq \sigma_1(n)$, then $\dot{r}_{q \ast \sigma_0}$ and $\dot{r}_{q \ast \sigma_1}$ are incomparable.
\end{claim}
\begin{proof}
We define by induction a fusion sequence $(p_m)_{m\in\omega}$ such that: $p_m = p$ for every $m \le n+1$; for every $m > n$ , for every $\sigma_0,\sigma_1 \in {}^{m}({}^{m} 2)$, if $\sigma_0(n) \neq \sigma_1(n)$, then $\dot{r}_{p_{m+1}\ast \sigma_0}$ and $\dot{r}_{p_{m+1}\ast \sigma_1}$ are incomparable.

Suppose we have defined $p_m$ with $m \ge n+1$, towards building $p_{m+1}$. Fix an enumeration $\{(\sigma_0^1, \sigma_1^1), \dots, (\sigma_0^h, \sigma_1^h)\}$ of the couples $(\sigma_0, \sigma_1)$ of elements of ${}^{m}({}^{m} 2)$ such that $\sigma_0 (n) \neq \sigma_1(n)$. By Claim~\ref{claim:1}, we can define a $\le_{m}$-descending sequence $(q_k)_{k \le h}$ such that $q_0 = p_m$ and, for every $0<k \le h$, $\dot{r}_{q_{k} \ast \sigma_0^k}$ and $\dot{r}_{q_{k} \ast \sigma_1^k}$ are incomparable. Set $p_{m+1} = q_h$. Then $p_{m+1} \le_{m} p_m$ and satisfies the desired property. 

Now let $q$ be the fusion of the $p_m$s. For every $m>n$, for every $\sigma_0,\sigma_1$ in ${}^m({}^m 2)$ with $\sigma_0(n) \neq \sigma_1(n)$, we have that $\dot{r}_{q\ast \sigma_0}$ and $\dot{r}_{q\ast \sigma_1}$ are incompatible. Indeed, $q \le_{m+1} p_{m+1}$, and since $\dot{r}_{p_{m+1}\ast \sigma_0}$ and $\dot{r}_{p_{m+1}\ast \sigma_1}$ are incomparable by construction, we are done.
\end{proof}
Now that we have proven Claim~\ref{claim:2}, fix some $q \le p$ that satisfies its statement.
\begin{claim}
$q \Vdash \dot{s}_n \le_c \dot{r}$.
\end{claim}
\begin{proof}
Let us show that
\begin{equation}\label{eq:claim3}
q \Vdash \forall m > n \ \forall \sigma \in {}^m({}^m 2) \ \big(\dot{r}_{q\ast \sigma} \subset \dot{r} \Longrightarrow  \dot{s}_n \in [(q \ast \sigma)(n)]\big).
\end{equation}
Once \eqref{eq:claim3} is proven, it follows that $q \Vdash \dot{s}_n \le_c \dot{r}$. Indeed, it directly follows from \eqref{eq:claim3} that
\[
q \Vdash \dot{s}_n = \bigcup\big\{q(n)(\sigma(n)) \mid \exists m > n \ (\sigma \in {}^m({}^m 2) \text{ and }\dot{r}_{q \ast \sigma} \subset \dot{r}\big\}.
\]

Suppose towards a contradiction that \eqref{eq:claim3} does not hold. Then, there exists some $z \le q$, some $m > n$ and $\sigma \in {}^m({}^m 2)$ such that 
\[
z \Vdash \dot{r}_{q\ast \sigma} \subseteq \dot{r} \text{ and } \dot{s}_n \not\in [(q \ast \sigma)(n)].
\]
By extending $z$ if necessary, we can assume that there exists a $\tau \in {}^m({}^m 2)$ such that $z \le q\ast\tau$. In particular, $\dot{r}_{q\ast \tau}\subseteq \dot{r}_{z}$. The statement $z \Vdash ``\dot{r}_{q\ast \sigma} \subseteq \dot{r}$'' is equivalent to $\dot{r}_{q\ast \sigma} \subseteq \dot{r}_{z}$. Thus, $\dot{r}_{q\ast \sigma}$ and $\dot{r}_{q\ast \tau}$ are comparable.  On the other hand, it follows from $z$ forcing $\dot{s}_n \not\in [(q \ast \sigma)(n)]$ that $\tau(n) \neq \sigma(n)$, and therefore, by the way we picked $q$, $\dot{r}_{q\ast \sigma}$ and $\dot{r}_{q\ast \tau}$ are incomparable. Contradiction.
\end{proof}
\end{proof}

\begin{corollary}\label{cor:realconstr}
Let $\alpha \in \kappa$ and $p\in\mathbb{S}^\kappa$ and let  $\dot{r}$ be a $\mathbb{S}^\kappa$-name for a real. The following are equivalent:
\begin{enumerate}[label={\upshape \arabic*)}]
\itemsep0.3em
\item\label{cor:realconstr-1}
$p \Vdash \dot{s}_\alpha \not\le_c \dot{r}$.
\item\label{cor:realconstr-2}
$p \Vdash \dot{r} \in \mathrm{L}[\dot{G} \upharpoonright (\kappa \setminus \{\alpha\})]$.
\item\label{cor:realconstr-3}
For any $q \le p$ there exists $z \le q$ such that for every $z_0, z_1 \le z$, if $z_0 \upharpoonright (\kappa\setminus \{\alpha\}) = z_1 \upharpoonright (\kappa\setminus \{\alpha\})$, then $\dot{r}_{z_0} = \dot{r}_{z_1}$.
\end{enumerate}
\end{corollary}
\begin{proof}
Directions $\text{\ref{cor:realconstr-3}}\Rightarrow \text{\ref{cor:realconstr-2}} \Rightarrow \text{\ref{cor:realconstr-1}}$ have already been shown in the proof for $\text{\ref{cor:realconstr-1}}\Rightarrow \text{\ref{cor:realconstr-2}} \Rightarrow \text{\ref{cor:realconstr-3}}$ of Proposition~\ref{prop:realconstr}.\vspace{0.3em}

$\text{\ref{cor:realconstr-1}} \Rightarrow \text{\ref{cor:realconstr-3}}$: Fix some $q \le p$. By hypothesis, $q \le p \Vdash \dot{s}_\alpha \not\le_c \dot{r}$. Hence, by Proposition~\ref{prop:realconstr}, $q$ does not force $\dot{r} \not\in \mathrm{L}[\dot{G} \upharpoonright (\kappa\setminus\{\alpha\})]$. In particular, there exists a $z \le q$ and a $\mathbb{S}^\kappa \upharpoonright (\kappa \setminus \{\alpha\})$-name $\dot{r}'$ such that $z$ forces $\dot{r} = \dot{r}'$. It quickly follows that $\dot{r}_{z_0} = \dot{r}_{z_1}$ for every $z_0, z_1 \le z$ with $z_0 \upharpoonright (\kappa\setminus\{\alpha\}) = z_1 \upharpoonright (\kappa\setminus\{\alpha\})$.
\end{proof}

The following result is key. It implies that any real in $\mathrm{L}[G]$ is equiconstructible with the set of Sacks reals constructible relative to it.

\begin{proposition}\label{prop:main}
In $\mathrm{L}[G]$, for every real $r$, the following hold:
\begin{enumerate}[label={\upshape \arabic*)}]
\itemsep0.3em
\item\label{prop:main1} For every $A \in [S]^{\le\omega}$, if $\{s \in S\mid s \le_c r\}\subseteq A$, then $r \le_c A$.
\item\label{prop:main2} $\{s \in S \mid s \le_c r\} \le_c r$.
\end{enumerate} 
In particular, $r \equiv_c \{s \in S \mid s \le_c r\}$.
\end{proposition}
\begin{proof}

By Lemma~\ref{lemma:general} and the remark afterward, we can suppose without loss of generality that $\kappa = \omega$.

Given a $\sigma \in {}^n({}^n 2)$ and some $m \le n$, we denote by $\sigma \upupharpoonright m$ the map whose domain is $m$ and such that $(\sigma \upupharpoonright m)(k) = \sigma (k) \upharpoonright m$ for every $k < m$. Note that $\sigma \upupharpoonright m$ is the natural projection of $\sigma$ onto ${}^m({}^m 2)$.

Fix an $\mathbb{S}^\omega$-name $\dot{r}$ for a real $r$ and fix a condition $p \in \mathbb{S}^\omega$. Before proving \ref{prop:main1} and \ref{prop:main2}, we first need to define in $\mathrm{L}$, by induction on $n$, a fusion sequence $(p_n)_{n\in\omega}$ along with an auxiliary map $\delta: \bigcup_{n\in\omega} {}^n({}^n 2) \rightarrow 2$ that satisfy the following properties:
\begin{enumerate}[label={\upshape \roman*)}]
\itemsep0.3em
\item\label{prop:main:cond1} $p_0 = p$,
\item\label{prop:main:cond2} For every $n \in \omega$ for every $\sigma \in {}^{n}({}^{n} 2)$, $p_{n+1}\ast \sigma$ decides $\dot{r} \upharpoonright n$,

\item\label{prop:main:cond3} For every $n \in \omega$, for every $\sigma \in {}^{n}({}^{n} 2)$, either:
\begin{enumerate}[label={\upshape iii\alph*)}, leftmargin=2.5\parindent]

\item\label{prop:main:cond3a} $p_{n+1} \ast \sigma \Vdash \dot{s}_{n} \le_c \dot{r}$, or
\item\label{prop:main:cond3b} for every $q_0, q_1 \le p_{n+1} \ast \sigma$, if $q_0 \upharpoonright (\omega\setminus \{n\}) = q_1 \upharpoonright (\omega\setminus \{n\})$, then $\dot{r}_{q_0} = \dot{r}_{q_1}$.
\end{enumerate} 
In the first case we set $\delta(\sigma)= 1$; otherwise we set $\delta(\sigma) = 0$,
\item\label{prop:main:cond4} For every $m,n\in\omega$ with $m<n$, for every $\sigma_0,\sigma_1 \in {}^{n}({}^{n} 2)$, if $\sigma_0(m)\neq \sigma_1(m)$ and $\sigma_0 \upupharpoonright m = \sigma_1 \upupharpoonright m$ and $\delta(\sigma_0\upupharpoonright m) = 1$, then $\dot{r}_{p_{n+1} \ast \sigma_0}$ and $\dot{r}_{p_{n+1} \ast \sigma_1}$ are incomparable. 
\end{enumerate}

Note that \ref{prop:main:cond3a} and \ref{prop:main:cond3b} are mutually exclusive, as, by Corollary~\ref{cor:realconstr}, \ref{prop:main:cond3b} implies $p_{n+1}\ast \sigma \Vdash \dot{s}_n \not\le_c \dot{r}$. In particular, the map $\delta$ is well-defined.

Now let us proceed with the inductive construction. Suppose we have defined $p_n$, towards defining $p_{n+1}$.  We do it in two steps: we first define a condition $q \le_{n} p_n$ that satisfies \ref{prop:main:cond2} and \ref{prop:main:cond3}. Then we define a condition $p_{n+1} \le_{n} q$ to take care of \ref{prop:main:cond4}. 

Fix an enumeration $\{\sigma_0, \dots, \sigma_h\}$ of $\sigma \in {}^{n}({}^{n} 2)$. It is straightforward, using Corollary~\ref{cor:realconstr}, to define a $\le_{n}$-decreasing sequence $(q_k)_{k\le h}$ such that $q_0 \le_{n} p_n$, and, for every $k \le h$, $q_k \ast \sigma_k$ decides $\dot{r}\upharpoonright n$ and satisfies either \ref{prop:main:cond3a} or \ref{prop:main:cond3b}. Given such a sequence, we let $q = q_h$. 

We now define the condition $p_{n+1} \le_{n} q$ that takes care of \ref{prop:main:cond4}. Fix an enumeration $\{(\sigma_0^0, \sigma_1^0, m_0), \dots, (\sigma^{j}_0, \sigma^{j}_1, m_j)\}$ of the triples $(\sigma_0, \sigma_1, m)$ with $\sigma_0, \sigma_1 \in {}^{n}({}^{n} 2)$ and $m < n$ such that
\[
\sigma_0\upupharpoonright m = \sigma_1\upupharpoonright m \text{ and } \sigma_0 (m) \neq \sigma_1(m) \text{ and }\delta(\sigma_0\upupharpoonright m) = 1.
\]
As before, we define a $\le_{n}$-descending sequence $(z_k)_{k \le j+1}$. Let $z_0 = q$. Fix some $k \le j$ and suppose we have constructed $z_k$, towards defining $z_{k+1}$.  Since $\delta(\sigma^k_0 \upupharpoonright m_k) = 1$, we know, by $\delta$'s definition, that $p_{m_k+1} \ast (\sigma^k_0\upupharpoonright m_k) \Vdash \dot{s}_{m_k}\le_c \dot{r}$. Now, since $z_{k} \le_{n} p_n \le_{m_k+1} p_{m_k+1}$, we have $z_k \ast (\sigma^k_0 \upupharpoonright m_k) \le p_{m_k+1} \ast (\sigma^k_0\upupharpoonright m_k)$. Therefore $z_k \ast (\sigma^k_0\upupharpoonright m_k) \Vdash \dot{s}_{m_k}\le_c \dot{r}$. Finally, we let $z_{k+1} \le_{n} z_k$ be such that $\dot{r}_{z_{k+1} \ast \sigma^k_0}$ and $\dot{r}_{z_{k+1} \ast \sigma^k_1}$ are incomparable---such a condition exists by Claim~\ref{claim:1} of Proposition~\ref{prop:realconstr}.

Let $p_{n+1}$ be $z_{j+1}$. By construction, $p_{n+1}$ satisfies conditions~\ref{prop:main:cond2}-\ref{prop:main:cond4}. This ends the inductive definition of the fusion sequence $(p_n)_{n\in\omega}$. Let $w$ be its fusion.  Going back to $\mathrm{L}[G]$, we can suppose that $w \in G$, by a density argument. Then, we claim the following (recall that we are assuming $\kappa = \omega$):
\begin{claim}\label{claim:main1}
For every $A\subseteq S$ with $\{s \in S \mid s \le_c r\}\subseteq A$, $r \le_c A$.
\end{claim}
\begin{proof}
Let $e_A: A\rightarrow \omega$ be as in the statement of Proposition~\ref{prop:definA}. For each $m\in\omega$, let $\bar{\sigma}_m$ be the unique element of ${}^m({}^m 2)$ such that $w \ast \bar{\sigma}_m  \in G$, or, equivalently, such that $s_k \in [(w \ast \bar{\sigma}_m)(k)]$ for every $k \in\omega$. We want to prove the following statement:
\begin{multline}\label{eq:3}
\forall m\in\omega \ \forall \sigma \in {}^m({}^m 2) \ \big(\forall k < m \ (k \in \text{ran}(e_A) \Rightarrow\\ \sigma(k) = \bar{\sigma}_m(k)) \Rightarrow \dot{r}_{w\ast \sigma} \subset r\big).
\end{multline}
Once we show \eqref{eq:3} we are done, as we would have (in $\mathrm{L}[G]$)
\begin{multline*}
r = \bigcup \Big\{\dot{r}_{w\ast \sigma} \bigm| \exists m\in\omega \ \big(\sigma \in {}^m({}^m 2) \text{ and } \\\forall k < m \ \big(k \in \mathrm{ran}(e_A) \Rightarrow e_A^{-1}(k) \in [(w \ast \sigma)(k)]\big)\big)\Big\}.
\end{multline*}

Indeed, note that, by condition~\ref{prop:main:cond2} of our fusion sequence, $\mathrm{length}(\dot{r}_{w\ast\sigma}) \ge m$ for every $\sigma \in {}^m({}^m 2)$. Since, by Proposition~\ref{prop:definA}, $e_A \in \mathrm{L}(A)$, we conclude that, by absoluteness, $r \in \mathrm{L}(A)$.

Towards showing that \eqref{eq:3} holds, fix an $m \in \omega$ and a $\sigma \in {}^m({}^m 2)$ such that for all $k<m$ with $k \in \text{ran}(e_A)$, $\sigma(k) = \bar{\sigma}_m(k)$. From the definition of $\bar{\sigma}_m$ it directly follows that $\dot{r}_{w\ast \bar{\sigma}_m} \subset r$. We now build a sequence $(\tau)_{i \le m}$ of elements of ${}^m({}^m 2)$ such that $\tau_0 = \sigma$, $\tau_{m} = \bar{\sigma}_m$  and $\dot{r}_{w \ast \tau_{i+1}} = \dot{r}_{w \ast \tau_{i}}$ for all $i<m$. This proves that $\dot{r}_{w\ast \sigma} = \dot{r}_{w\ast \bar{\sigma}_m} \subset r$.

For each $i < m$, let $\tau_{i+1}$ be defined as the element of ${}^m({}^m 2)$ such that $\tau_{i+1}(k) = \bar{\sigma}_m(k)$ for all $k \le i$, and $\tau_{i+1}(k) = \sigma(k)$ for all $k>i$.

Fix an $i<m$, and suppose that $\tau_{i+1} \neq \tau_i$. By $\tau_i$'s definition, the only coordinate on which $\tau_{i+1}$ and $\tau_i$ can differ is the $i$-th. In particular, $\bar{\sigma}_m(i) = \tau_{i+1}(i) \neq \tau_i(i) = \sigma(i)$, and $\tau_{i+1} \upharpoonright (m\setminus\{i\}) = \tau_{i} \upharpoonright (m\setminus\{i\})$. By the way we picked $\sigma$, it must be the case that $i\not\in \text{ran}(e_A)$, or, equivalently, that $s_i \not\in A$. In particular, by our hypothesis on $A$, $s_{i} \not\le_c r$. Thus, $\delta(\bar{\sigma}_i) = 0$, as otherwise, by $\delta$'s definition, $p_{i+1} \ast \bar{\sigma}_i$ would force $\dot{s}_i \le_c \dot{r}$, but this is impossible, as $p_{i+1}\ast\bar{\sigma}_i$, which is extended by $w\ast \bar{\sigma}_i$, belongs to the the generic filter $G$.

Note that both $w\ast \tau_i$ and $w\ast \tau_{i+1}$ extend $p_{i+1}\ast \bar{\sigma}_i$, and that $(w\ast \tau_i) \upharpoonright (\omega\setminus \{i\}) = (w\ast \tau_{i+1}) \upharpoonright (\omega\setminus \{i\})$. By the fact that $\delta(\bar{\sigma}_i) = 0$, we conclude that $\dot{r}_{w\ast\tau_i} = \dot{r}_{w\ast \tau_{i+1}}$ and we are done.
\end{proof}
\begin{claim}\label{claim:main2}
$ \{s \in S \mid s \le_c r\} \le_c r$.
\end{claim}
\begin{proof}

Let the $\bar{\sigma}_m$s be defined as the beginning of proof of Claim~\ref{claim:main1}. We want to show the following:

\begin{equation}\label{eq:4}
  \tag{$\ddagger$}
  \parbox{\dimexpr\linewidth-10em}{
    For every $m \in\omega$, for every $\sigma \in {}^m({}^m 2)$ such that $\dot{r}_{w \ast \sigma} \subset r$, the following holds:
    \begin{enumerate}[label={$(\ddagger)_{\arabic*}$}]
    \itemsep0.3em
    \item For all $k < m$, if $s_k \le_c r$, then $\sigma(k) = \bar{\sigma}_m(k)$.
    \item $s_m \le_c r$ if and only if $\delta(\sigma) = 1$.
    \end{enumerate}
}
\end{equation}

\noindent Once we show \eqref{eq:4} we are done. Indeed, it directly follows from \eqref{eq:4} that, in $\mathrm{L}[G]$,
\begin{multline*}
\{s \in S \mid s \le_c r\} = \Big\{\bigcup\{w(k)(\tau(k)) \mid m > k \text{ and } \tau \in {}^m({}^m 2) \text{ and } \dot{r}_{w \ast \tau} \subset r\} \\ \bigm| k \in \omega \text{ and } \exists \sigma \in {}^k ({}^k 2) \ (\delta(\sigma) = 1 \text{ and } \dot{r}_{w \ast \sigma} \subset r)\Big\}
\end{multline*}

As $w$, $\dot{r}$ and $\delta$ are all constructible, then, by absoluteness, we conclude that $\{s \in S \mid s \le_c r\} \in \mathrm{L}[r]$.

Towards proving \eqref{eq:4}, fix an $m\in\omega$ and $\sigma \in {}^m({}^m 2)$ such that $\dot{r}_{w\ast \sigma} \subset r$. We first want to prove $\eqref{eq:4}_1$. Suppose towards a contradiction that there exists a $k< m$ such that $s_k \le_c r$ and $\sigma(k)\neq \bar{\sigma}_m(k)$. Let $\ubar{k}$ be the least such. Let $\tau \in  {}^m({}^m 2)$ be defined as follows: for every $k< m$, if $ k < \ubar{k}$, then $\tau(k) = \sigma(k)$; otherwise, $\tau(k) = \bar{\sigma}_m(k)$.

We want to show that $\dot{r}_{w\ast \tau} \subset r$ and, moreover, that $\dot{r}_{w\ast \tau}$  and $\dot{r}_{w\ast \sigma}$ are incomparable, which would lead to a contradiction, as we are assuming $\dot{r}_{w\ast \sigma} \subset r$. 

By the minimality of $\ubar{k}$, we have that $\tau(k) = \bar{\sigma}_m(k)$ for every $k < m$ such that $s_k \le_c r$. Thus, by \eqref{eq:3} with $A = \{s \in S \mid s \le_c r\}$, we conclude $\dot{r}_{w\ast \tau} \subset r$. So we are left to show that $\dot{r}_{w\ast \tau}$  and $\dot{r}_{w\ast \sigma}$ are incomparable in order to reach the desired contradiction. Note that, by the definition of $\tau$, $\tau \upupharpoonright \ubar{k} = \sigma \upupharpoonright \ubar{k}$. 

\begin{subclaim}\label{subclaim:main2}
$\delta(\sigma \upupharpoonright \ubar{k}) = 1$.
\end{subclaim}
\begin{proof}
We prove the subclaim using an argument analogous to the one used to prove \eqref{eq:3} and show that $\delta(\sigma \upupharpoonright \ubar{k}) = \delta(\bar{\sigma}_{\ubar{k}})$. Then, as $s_{\ubar{k}} \le_c r$, it must be that $\delta(\bar{\sigma}_{\ubar{k}}) = 1$, and we conclude $\delta(\sigma \upupharpoonright \ubar{k}) = \delta(\bar{\sigma}_{\ubar{k}})  = 1$ as desired.

We build a sequence $(\rho_i)_{i \le \ubar{k}}$ of elements of ${}^\ubar{k}({}^\ubar{k} 2)$ such that $\rho_0 = \sigma \upupharpoonright \ubar{k}$, $\rho_{\ubar{k}} = \bar{\sigma}_{\ubar{k}}$  and $\delta(\rho_i) = \delta(\rho_{i+1})$ for all $i<\ubar{k}$. This proves that $\delta(\sigma \upupharpoonright \ubar{k}) = \delta(\bar{\sigma}_{\ubar{k}})$.

For each $i < \ubar{k}$, let $\rho_{i+1}$ be defined as the element of ${}^\ubar{k}({}^\ubar{k} 2)$ such that $\rho_{i+1}(k) = \bar{\sigma}_{\ubar{k}}(k)$ for all $k \le i$, and $\rho_{i+1}(k) = \sigma(k) \upharpoonright \ubar{k}$ for all $k>i$.

Fix an $i<\ubar{k}$, and suppose that $\rho_{i+1} \neq \rho_i$. By $\rho_i$'s definition, the only coordinate on which $\rho_{i+1}$ and $\rho_i$ can differ is the $i$-th. In particular, $\bar{\sigma}_{\ubar{k}}(i) = \rho_{i+1}(i) \neq \rho_i(i) = \sigma(i) \upharpoonright \ubar{k}$, and $\rho_{i+1} \upharpoonright (\ubar{k}\setminus\{i\}) = \rho_{i} \upharpoonright (\ubar{k}\setminus\{i\})$. By the minimality of $\ubar{k}$, it must be the case that $s_{i} \not\le_c r$. Thus, $\delta(\bar{\sigma}_i) = 0$, as otherwise, by $\delta$'s definition, $p_{i+1} \ast \bar{\sigma}_i$ would force $\dot{s}_i \le_c \dot{r}$, but this is impossible, as $p_{i+1}\ast\bar{\sigma}_i$, which is extended by $w\ast \bar{\sigma}_i$, belongs to the the generic filter $G$.

By the fact that $\delta(\bar{\sigma}_i) = 0$, and by $\delta$'s definition, we conclude that there exists a $\mathbb{S}^\omega \upharpoonright (\omega\setminus \{i\})$-name $\dot{r}'$ such that $p_{i+1} \ast \bar{\sigma}_i \Vdash \dot{r} = \dot{r}'$. In particular, for every condition $q \le p_{i+1} \ast \bar{\sigma}_i$, the following holds by absoluteness of the constructibility preorder,
\[
q \Vdash \dot{s}_\ubar{k} \le_c \dot{r}' \iff q \upharpoonright (\omega \setminus \{i\}) \Vdash \dot{s}_\ubar{k} \le_c \dot{r}'.
\]
Thus, since both $p_{\ubar{k}+1} \ast \rho_i$ and $p_{\ubar{k}+1} \ast \rho_{i+1}$ extend $p_{i+1} \ast \bar{\sigma}_i$ and since $(p_{\ubar{k}+1} \ast \rho_i) \upharpoonright (\omega \setminus \{i\}) = (p_{\ubar{k}+1} \ast \rho_{i+1}) \upharpoonright (\omega \setminus \{i\})$, we conclude $\delta(\rho_i) = \delta(\rho_{i+1})$.
\end{proof}

By Subclaim~\ref{subclaim:main2} and by condition~\ref{prop:main:cond4} of the fusion sequence, we know that $\dot{r}_{p_{m+1} \ast \tau}$ and $\dot{r}_{p_{m+1} \ast \sigma}$ are incomparable. But since $w \le_{m+1} p_{m+1}$, we have $\dot{r}_{p_{m+1} \ast \tau} \subseteq \dot{r}_{w \ast \tau} $ and $\dot{r}_{p_{m+1} \ast \sigma} \subseteq \dot{r}_{w \ast \sigma} $. Therefore, $\dot{r}_{w \ast \tau}$ and $\dot{r}_{w \ast \sigma}$ are incomparable. We have reached the desired contradiction, and we conclude that  $\eqref{eq:4}_1$ holds.

Now we want to show $\eqref{eq:4}_2$. Using $\eqref{eq:4}_1$, and an argument analogous to the one used in the proof of Subclaim~\ref{subclaim:main2}, we can show that $\delta(\sigma) = \delta(\bar{\sigma}_m)$. Indeed, note that by $\eqref{eq:4}_1$, $\sigma(k) \neq \bar{\sigma}_m(k)$ implies $s_k \not\le_c r$ for every $k < m$.  But then, since $s_m \le_c r$ holds if and only if $\delta(\bar{\sigma}_m)= 1$, we conclude that $s_m \le_c r$ holds if and only if $\delta(\sigma) = 1$.
\end{proof}
\end{proof}

\begin{proof}[Proof of Theorem~{\upshape\ref{thm:representation}}]
Consider the following map:
\begin{align*}
\Omega: (\mathcal{D}_c, \le_c) &\longrightarrow (\mathcal{R}, \subseteq)\\ \mathbf{r} &\longmapsto \{\alpha\in\kappa \mid \mathbf{s}_\alpha \le_c \mathbf{r}\}.
\end{align*}
We claim that $\Omega$ is an isomorphism. But first, we need to show that it is well-defined, i.e., that the range of $\Omega$ is a subset of $\mathcal{R}$. Fix an $\mathbf{r} \in \mathcal{D}_c$. It immediately follows from Lemma~\ref{lemma:general} that $\Omega(\mathbf{r}) \in [\kappa]^{\le\omega}$. Now fix an $\alpha \in \kappa$ such that $\mathbf{s}_\alpha \le_c \Omega(\mathbf{r})$ towards showing that $\alpha \in \Omega(\mathbf{r})$.  Let $S_\mathbf{r} \coloneqq \{s \in S \mid \mathbf{s} \le_c \mathbf{r}\}$. Note that $\Omega(\mathbf{r}) = \mathrm{ran}(e_{S_\mathbf{r}})$, where the map $e$ is the one defined in Proposition~\ref{prop:definA}. By \ref{prop:main2} of Proposition~\ref{prop:main}, $S_\mathbf{r} \le_c \mathbf{r}$. We already noted that $S_\mathbf{r}$ must be countable. Finally, from \ref{prop:definA-1} of Proposition~\ref{prop:definA} it follows that $\Omega(\mathbf{r}) = \mathrm{ran}(e_{S_\mathbf{r}}) \le_c S_\mathbf{r}$. Thus, $\Omega(\mathbf{r}) \le_c S_\mathbf{r} \le_c \mathbf{r}$. Since, by assumption, $\mathbf{s}_\alpha \le_c \Omega(\mathbf{r})$, we conclude that $\mathbf{s}_\alpha \le_c \mathbf{r}$ and hence $\alpha \in \Omega(\mathbf{r})$ by $\Omega$'s definition. Hence, $\Omega$ is well-defined. 

\begin{claim}
$\Omega$ is injective.
\end{claim}
\begin{proof}
Fix two reals $a, b$ and suppose that $\Omega(\mathbf{a}) = \Omega(\mathbf{b})$. In particular, $S_\mathbf{a} = S_\mathbf{b}$ (see the definition of $S_\mathbf{r}$ in the previous paragraph) and, by Proposition~\ref{prop:main}, $a \equiv_c S_\mathbf{a} = S_\mathbf{b} \equiv_c b$. Thus, $\mathbf{a} = \mathbf{b}$.
\end{proof}

\begin{claim}
$\Omega$ is surjective.
\end{claim}
\begin{proof}
Fix an $x \in \mathcal{R}$. Let $A = \{s_\alpha \mid \alpha \in x\}$ and fix a real $r$ such that $r \equiv_c A$---note that such a real exists by Corollary~\ref{cor:realset}. We now prove that $\Omega(\mathbf{r}) = x$.  Clearly,  $x\subseteq \Omega(\mathbf{r})$. Now, fix an $\alpha\not\in x$ towards showing that $\alpha \not\in\Omega(\mathbf{r})$, or, equivalently, that $s_\alpha \not\le_c A$. Since $\alpha\not\in x$, we have $s_\alpha \not\le_c x$, by $\mathcal{R}$'s definition. It follows from Corollary~\ref{cor:realconstr} that $x \in \mathrm{L}[G \upharpoonright (\kappa\setminus\{\alpha\})]$. Therefore, also $A$ belongs to $\mathrm{L}[G \upharpoonright (\kappa\setminus\{\alpha\})]$, and we conclude that $s_\alpha \not\le_c A$.
\end{proof}

With the following claim, we are done. Indeed, it implies that $\Omega$ is a join-semilattice homomorphism. Moreover, as a by-product, we get that $(\mathcal{R}, \cup)$ is a join-semilattice.
\begin{claim}
For all $\mathbf{a}, \mathbf{b} \in \mathcal{D}_c$, $\Omega(\mathbf{a} \oplus \mathbf{b}) = \Omega(\mathbf{a}) \cup \Omega(\mathbf{b})$.
\end{claim}
\begin{proof}
Clearly, $\Omega(\mathbf{a}) \cup \Omega(\mathbf{b}) \subseteq \Omega(\mathbf{a} \oplus \mathbf{b})$. Now fix an $\alpha \not\in \Omega(\mathbf{a}) \cup \Omega(\mathbf{b})$ towards showing that $\alpha \not\in \Omega(\mathbf{a} \oplus \mathbf{b})$. By $\Omega$'s definition, $\mathbf{s}_\alpha \not\le_c \mathbf{a},\mathbf{b}$. By Corollary~\ref{cor:realconstr}, both $\mathbf{a}$ and $\mathbf{b}$ belong to $\mathrm{L}[G \upharpoonright (\kappa\setminus\{\alpha\})]$. Therefore, $\mathbf{a} \oplus \mathbf{b}$ also belong to  $\mathrm{L}[G \upharpoonright (\kappa\setminus\{\alpha\})]$, and we conclude that $\mathbf{s}_\alpha \not\le_c \mathbf{a} \oplus \mathbf{b}$. Thus, $\alpha \not\in \Omega(\mathbf{a} \oplus \mathbf{b})$.
\end{proof}
\end{proof}

\section{Proofs of Theorems~\ref{thm:main1}, \ref{thm:main2} and \ref{thm:main3}}\label{sec:mainproofs}

In this section, we prove our main results, starting with Theorem~\ref{thm:main1}. As before, we fix a generic filter $G$ for $\mathbb{S}^\kappa$ over $\mathrm{L}$.

Recall that a poset $(P, \le)$ is said to be \emph{complemented} if it is bounded and, for every $x \in P$, there is some (not necessarily unique) $y$ such that $x \wedge y = \mathbf{0}_P$ and $x \vee y = \mathbf{1}_P$.

\begin{proof}[Proof of Theorem~{\upshape\ref{thm:main1}}]
By Theorem~\ref{thm:representation}, we can substitute $(\mathcal{D}_c, \le_c)$ with $(\mathcal{R}, \subseteq)$.

Let $s_0^+ \coloneqq s_0 \cup \{0\}$ and $s_0^- \coloneqq s_0 \setminus \{0\}$. We first claim that $s_0^+ \in \mathcal{R}$ and $s_0^- \not\in \mathcal{R}$. Clearly, $s_0^+ \equiv_c s_0^- \equiv_c s_0$. Therefore, by mutual genericity of the $s_\alpha$'s, we have that $s_0$ is the only element of $S$ which is constructible from $s_0^+$ and $s_0^-$. Since, by definition, $0 \in s_0^+$ and $0 \not\in s_0^-$, we conclude that $s_0^+ \in \mathcal{R}$ and $s_0^- \not\in\mathcal{R}$. Now we proceed with the proof of the theorem. 

We claim that in $(\mathcal{R}, \subseteq)$ the set $\{s_0^+, \omega\setminus \{0\}\}$ does not have a greatest lower bound. 
Suppose towards a contradiction that such a greatest lower bound exists, and let us name it $K$. We want to show that the following holds:
\begin{equation}\label{eq:thm1}
s_0^- = \bigcup_{n \in s_0^-} \{n\}  \subseteq K\subseteq s_0^+ \cap (\omega\setminus\{0\}) = s_0^-.
\end{equation}
The first $\subseteq$ follows from having assumed $K$ to be the greatest lower bound of $\{s_0^+, \omega\setminus\{0\}\}$, together with the fact that $\{n\} \in \mathcal{R}$ and $\{n\} \subseteq s_0^+, \omega\setminus \{0\}$ for every $n \in s_0^-$; the second $\subseteq$ follows from $K$ being a lower bound of $\{s_0^+, \omega\setminus\{0\}\}$. It directly follows from \eqref{eq:thm1} that $K = s_0^-$. However, $s_0^-$ does not belong to $\mathcal{R}$, hence the contradiction.

We claim that in $(\mathcal{R}, \subseteq)$ the set $\{\{n\} \mid n \in s_0^-\}\subset \mathcal{R}$ does not have a least upper bound. If we suppose towards a contradiction that such a least upper bound exists, and we name it $K$, then \eqref{eq:thm1} would still hold for reasons analogous to the ones employed in the previous paragraph. Therefore, we would reach the same contradiction, as $K$ would coincide with $s_0^-$, which does not belong to $\mathcal{R}$. 

We now prove that $(\mathcal{R}, \subseteq)$ is not complemented. If $\kappa$ is uncountable, then, since $\kappa$ is still uncountable in $\mathsf{L}[G]$, the claim trivially follows as $(\mathcal{R}, \subseteq)$ is not bounded above. Hence suppose $\kappa = \omega$. Since every singleton $\{n\}$, with $n \in\omega$, belongs to $\mathcal{R}$, we must prove that there exists some $x \in \mathcal{R}$ such that $\omega\setminus x \not\in \mathcal{R}$. We have already proven that $s_0^+ \in \mathcal{R}$. Moreover, $\omega\setminus s_0^+ \not\in \mathcal{R}$, as $s_0 \equiv_c s_0^+ \equiv_c \omega\setminus s_0^+$, while $0 \not\in \omega\setminus s_0^+$. 
\end{proof}

We need the following couple of lemmas before proving Theorem~\ref{thm:main2}.

\begin{lemma}\label{lemma:proper}
In $\mathrm{L}[G]$, for every map $f: \kappa \rightarrow \kappa$ there exists a constructible, constructibly countable $D\subseteq \kappa$ closed under $f$ such that $f \upharpoonright D \in \mathrm{L}[G \upharpoonright D]$.
\end{lemma}
\begin{proof}
The proof is very similar to the one of Lemma~\ref{lemma:general}. We work in $\mathrm{L}$. Fix some $p \in \mathbb{S}^\kappa$ and some $\mathbb{S}^\kappa$-name $\dot{f}$ such that 
\[
p \Vdash \dot{f}:\kappa \rightarrow \kappa \text{ is a map}.
\]

Via a simple bookkeeping argument, define a sequence $(p_n, F_n)_{n\in\omega}$ such that $(p_n)_{n\in\omega}$ is a fusion sequence witnessed by $(F_n)_{n\in\omega}$ with $p_0 = p$ and such that for every $n \in\omega$, for every $\sigma \in {}^{F_n}({}^n 2)$, for every $\alpha \in F_n$, there exists some $\beta \in F_{n+1}$ such that $p_{n+1} \ast \sigma \Vdash \dot{f}(\alpha) = \beta$.

Let $q$ be the fusion of the $p_n$s and let $D$ be its support. Then, by construction, $q$ forces $D$ to be closed under $\dot{f}$ and $\dot{f} \upharpoonright D \in \mathrm{L}[\dot{G} \upharpoonright D]$. By density, we are done.
\end{proof}

\begin{lemma}\label{lemma:rigid}
In $\mathrm{L}[G]$, for every bijection $\psi:\kappa\rightarrow \kappa$ there exists a constructible sequence $(A_n)_{n\in\omega}$ of pairwise disjoint, finite subsets of $\kappa$, and a constructible sequence $(\alpha_n)_{n\in\omega}$ of ordinals in $\kappa$ such that $\psi(\alpha_n) \in A_n$ for every $n$.
\end{lemma}

\begin{proof}
By Lemma~\ref{lemma:proper}, we can assume without loss of generality that $\kappa = \omega$. Fix some $p \in G$ and a $\mathbb{S}^\omega$-name $\dot{\psi}$ for $\psi$ such that 
$$
p \Vdash \dot{\psi}:\omega\rightarrow \omega \text{ is a bijection}.
$$

Working in $\mathrm{L}$, we inductively define a fusion sequence $(p_n)_{n\in\omega}$ and a sequence $(A_n)_{n\in\omega}$ of finite subsets of $\omega$ and a sequence $(\alpha_n)_{n\in\omega}$ of positive integers such that:
\begin{enumerate}[label=\roman*)]
\itemsep0.3em
\item\label{:rigid:0} $p_0 = p$.
\item\label{:rigid:1} For every $n\in\omega$, for every $\sigma \in {}^n({}^n 2)$,  $p_{n+1} \ast \sigma \Vdash \dot{\psi}(\alpha_{n}) \in A_{n}$.
\item\label{:rigid:2} For every $n \in\omega$, for every $\sigma \in {}^n({}^n 2)$, for every $k \in A_{n}$,  $p_{n+1} \ast \sigma$ decides the value of $\dot{\psi}^{-1}(k)$.
\item\label{:rigid:3} For all distinct $n,m \in \omega$, $A_n \cap A_{m} = \emptyset$.
\end{enumerate}

Now, we describe the inductive construction. Fix an $n\in\omega$ and suppose we defined $p_{m}$ for all $m \le n$ and $A_m, \alpha_m$ for all $m < n$ towards defining $p_{n+1}, A_{n}$ and $\alpha_{n}$. Let 
\[
E_n \coloneqq \Big\{\alpha \in \omega \mid \exists m < n \ \exists \tau \in {}^{m}({}^{m} 2) \  \exists k \in A_m \ \big(p_{m+1} \ast \tau \Vdash \dot{\psi}(\alpha) = k\big) \Big\}.
\]
Let $\alpha_{n}$ be any positive integer not in $E_n$---note that this is possible as $E_n$ is finite. Now, fix an enumeration $\{\sigma_0, \dots, \sigma_h\}$ of ${}^{n}({}^{n} 2)$. It is routine to define a $\le_{n}$-decreasing sequence $(q_i)_{i\le h}$ such that $q_0 \le_{n} p_n$, and $q_i \ast \sigma_i$ decides $\dot{\psi}(\alpha_{n})$ for every $i \le h$. Given such a sequence, we let  \[
A_{n} \coloneqq \Big\{k \in \omega \mid \exists i \le h \ \big(q_i \ast \sigma_i \Vdash \dot{\psi}(\alpha_{n}) = k\big)\Big\}.
\]

We claim that $A_{n} \cap A_m = \emptyset$ for every $m < n$. Indeed, suppose towards a contradiction that there exists some $m < n$ with $A_{n} \cap A_m \neq \emptyset$, and fix a $k \in A_{n} \cap A_m$. By definition of $A_{n}$, there exists an $i \le h$ such that $q_i \ast \sigma_i \Vdash \dot{\psi}(\alpha_n) = k$.  By construction, $q_i \le_{m+1} p_{m+1}$, and therefore $q_i \ast \sigma_i \le p_{m+1} \ast (\sigma_i \upupharpoonright m)$. By condition~\ref{:rigid:2}, $p_{m+1} \ast (\sigma_i \upupharpoonright m)$ already decides $\dot{\psi}^{-1}(k)$. Thus, it follows that $p_{m+1} \ast (\sigma_i \upupharpoonright m) \Vdash \dot{\psi}(\alpha_n) = k$. By definition of $E_n$, $\alpha_n \in E_n$, which is a contradiction, as we picked $\alpha_{n}$ outside of $E_n$. 

Finally, defining a $p_{n+1} \le_{n} q_h$ that satisfies condition~\ref{:rigid:2} is routine.

The sequences defined in this way satisfy \ref{:rigid:0}-\ref{:rigid:3}. Let $z$ be the fusion of the $p_n$s. Then $z$ extends $p$ and forces our sequences $(A_n)_{n \in \omega}$ and $ (\alpha_n)_{n\in\omega}$ to have the desired properties. By density, we are done.
\end{proof}
\begin{proof}[Proof of Theorem~{\upshape\ref{thm:main2}}]
By Theorem~\ref{thm:representation}, we can equivalently prove that $(\mathcal{R}, \subseteq)$ is rigid in $\mathrm{L}[G]$. 

Every automorphism $f: \mathcal{R} \rightarrow \mathcal{R}$ is canonically induced by a bijection $\psi: \kappa\rightarrow \kappa$ such that, for every $x \in [\kappa]^{\le \omega}$, $x \in \mathcal{R}$ if and only if $\psi[x] \in \mathcal{R}$.  So let us assume towards a contradiction that there exists a bijection $\psi:\kappa\rightarrow\kappa$ such that $\psi \neq \text{id}$ and, for every $x \in [\kappa]^{\le\omega}, \ x \in \mathcal{R}$ if and only if $\psi[x] \in \mathcal{R}$. 

Fix one such $\psi$, and assume $\psi(0) = 1$ just for the sake of simplicity. Fix the constructible sequences $(A_n)_{n\in\omega}$ and $(\alpha_n)_{n\in\omega}$ given by Lemma~\ref{lemma:rigid} for our $\psi$. Since the $A_n$s are mutually disjoint, we can assume without loss of generality that $0,1 \not\in A_n$ for any $n\in\omega$.

We now define (in $\mathrm{L}[G]$) an $r \in \mathcal{R}$ such that $1 \not\in r$ and $s_0 \le_c \psi^{-1}(r)$.  To see why this leads to a contradiction, note the following:  $\psi(0) = 1$ by assumption; therefore, as $1 \not\in r$, we have $0\not\in \psi^{-1}(r)$; moreover, since $\psi^{-1}(r)\in \mathcal{R}$ and $s_0 \le_c \psi^{-1}(r)$, it follows that $0 \in \psi^{-1}(r)$, and hence the contradiction. 

We are now ready to define $r$. Let 
$$
r \coloneqq \{0\} \cup \bigcup_{n \in s_0} A_n.
$$
Note that $1\not\in r$. As the $A_n$'s are finite, $r \in [\kappa]^{\le \omega}$. Moreover, since the sequence $(A_n)_{n\in\omega}$ is constructible and the $A_n$s are mutually disjoint, we have $r \equiv_c s_0$. This implies that $r \in \mathcal{R}$, as $0 \in r$ by definition of $r$. Finally, note that $s_0$ is constructible relative to $\psi^{-1}(r)$, as the following holds for every $n\in \omega$:
\begin{align*}
n \in s_0 &\iff  A_n \subseteq r\\ &\iff \psi^{-1}(A_n) \subseteq \psi^{-1}(r) \\ &\iff \alpha_n \in \psi^{-1}(r), 
\end{align*}
where the first equivalence comes directly from the definition of $r$ and the almost-disjointness of the $A_n$'s, and the last equivalence follows from the properties of the constructible sequence $(\alpha_n)_{n\in\omega}$. Hence, we reach the contradiction described in the previous paragraph.
\end{proof}

A priori, it could be the case that the rigidity of $(\mathcal{D}_c, \le_c)$ follows from the definability of many degrees in $(\mathcal{D}_c, \le_c)$. However, Theorem~\ref{thm:main3} tells us that this is not the case. Before delving into the proof of the theorem, we need the following well-known lemma.
\begin{lemma}[Symmetry Lemma, {\cite[Lemma 14.37]{MR1940513}}]
Let $\mathbb{P}$ be a forcing notion, $ \pi \in \mathrm{Aut} ( \mathbb{P} ) $ and $ \dot{x}_1 , \dots , \dot{x}_n $ be $\mathbb{P}$-names. 
For every formula \( \varphi ( x_1, \dots , x_n ) \)
$$
p \Vdash \varphi ( \dot{x}_1 , \dots , \dot{x}_n ) \iff \pi (p) \Vdash \varphi ( \pi (\dot{x}_1) , \dots , \pi( \dot{x}_n) ) .
$$
\end{lemma}

Recall that, given a forcing $\mathbb{P}$, every automorphism $\pi \in \mathrm{Aut} ( \mathbb{P} )$ acts canonically on $\mathbb{P}$-names as follows: for every $\mathbb{P}$-name $\dot{x}$,
\begin{equation}\label{eq:defaut}
 \pi \dot{x} = \{ ( \pi \dot{y} , \pi p ) \mid ( \dot{y} , p ) \in \dot{x}\}.
\end{equation} 

\begin{proof}[Proof of Theorem~{\upshape\ref{thm:main3}}]
First note that $(\mathcal{D}_c, \le_c)$ has a greatest element only when $\kappa=\omega$. This fact, although being an easy consequence of Lemma~\ref{lemma:general}, follows directly from Theorem~\ref{thm:representation}.  

In order to prove Theorem~\ref{thm:main3}, we can use  Theorem~\ref{thm:representation}, and equivalently prove that no element of $\mathcal{R}$ other than $\emptyset$ (and $\omega$, in case $\kappa = \omega$) is definable in the structure $(\mathcal{R}, \subseteq)$. 

In $\mathrm{L}[G]$, fix some nonempty $a\in \mathcal{R}$ with $a \neq \kappa$ and a formula $\phi(x)$ without parameters such that $(\mathcal{R}, \subseteq) \vDash \phi(a)$. We want to find a $b \in \mathcal{R}$ with $a \neq b$ such that $(\mathcal{R}, \subseteq) \vDash \phi(b)$, thus showing that no parameter-free formula can pick out a unique element of $\mathcal{R}$ other than $\emptyset$ (and $\omega$, in case $\kappa = \omega$) in $(\mathcal{R}, \subseteq)$. 

Fix some $p \in G$ and a $\mathbb{S}^\kappa$-name $\dot{a}$ for $a$ such that 
\begin{equation}\label{eq:99}
p \Vdash (\dot{\mathcal{R}}, \subseteq) \vDash \phi(\dot{a}),
\end{equation}
where $\dot{\mathcal{R}}$ is the $\mathbb{S}^\kappa$-name for $\mathcal{R}$. Moreover, since $a$ is nonempty and different from $\kappa$, we may assume that
\[
p \Vdash 0 \in \dot{a} \text{ and } 1 \not\in \dot{a},
\]
just for the sake of simplicity.  

From now on we work in $\mathrm{L}$. We want to define an automorphism $\sigma \in \mathrm{Aut}(\mathbb{S}^\kappa \downarrow p)$---i.e., an automorphism on the principal ideal of $p$ in $\mathbb{S}^\kappa$---such that $p \Vdash ``\dot{s}_0 \equiv_c \sigma(\dot{s}_1)$'' and $p \Vdash ``\dot{s}_1 \equiv_c \sigma(\dot{s}_0)$''. We define it as follows: let $h: [p(0)] \rightarrow [p(1)]$ be the canonical homeomorphism between $[p(0)]$ and $[p(1)]$; given some $q \in \mathbb{S}^\kappa$ with $q \le p$, we let $\sigma(q)$ be the condition defined by:
\[
\forall \alpha \in \kappa, \ \ \sigma(q)(\alpha) \coloneqq \begin{cases}
q(\alpha) &\text{ if } \alpha\neq 0,1\\
\big\{t \in {}^{<\omega}2 \mid h[N_t] \cap [q(1)] \neq \emptyset\big\} &\text{ if } \alpha=0\\
\big\{t \in {}^{<\omega}2 \mid h^{-1}(N_t) \cap [q(0)] \neq \emptyset\big\} &\text{ if } \alpha=1
\end{cases}
\]
By construction, $\sigma(p) = p$. Moreover, it follows from the definition of $\sigma$ that $p$ forces $\sigma(\dot{s}_0) = h^{-1}(\dot{s}_1)$ and $\sigma(\dot{s}_1) = h(\dot{s}_0)$. Since the homeomorphism $h$ is constructibly coded, we have $p \Vdash ``\dot{s}_0 \equiv_c \sigma(\dot{s}_1)$'' and $p \Vdash ``\dot{s}_1 \equiv_c \sigma(\dot{s}_0)$'' as desired. Moreover, $\sigma(\dot{s}_\alpha) = \dot{s}_\alpha$ for every $\alpha > 1$.

Let $\theta: \kappa\rightarrow \kappa$ be the bijection that simply swaps $0$ and $1$, leaving every other ordinal in $\kappa$ fixed. Clearly, $\theta \circ \theta = \mathrm{id}$. Let $\dot{f}$ be the $\mathbb{S}^\kappa$-name for the function that maps every $x \in [\kappa]^{\le \omega}$ to $\theta[x]$. Clearly, $\Vdash \dot{f} = \dot{f}^{-1}$. We claim that
\begin{equation}\label{eq:100}
p \Vdash \dot{f} \upharpoonright \dot{\mathcal{R}} \text{ is an isomorphism from } (\dot{\mathcal{R}}, \subseteq) \text{ to } (\sigma(\dot{\mathcal{R}}), \subseteq).
\end{equation}

In order to show that \eqref{eq:100} holds, let us analyze how $\sigma$ acts on the name $\dot{\mathcal{R}}$. By definition of $\mathcal{R}$, we have
\[
p \Vdash \forall x \big(x \in \dot{\mathcal{R}} \text{ iff } x \in [\kappa]^{\le\omega} \text{ and } \forall \alpha \in \kappa \ (\dot{s}_\alpha \le_c x \Rightarrow \alpha \in x)\big).
\]
By the Symmetry Lemma,
\[
\sigma(p) = p \Vdash \forall x \big(x \in \sigma(\dot{\mathcal{R}}) \text{ iff } x \in [\kappa]^{\le\omega} \text{ and } \forall \alpha \in \kappa \ (\sigma(\dot{s}_\alpha) \le_c x \Rightarrow \alpha \in x)\big),
\]
But since $p$ forces $\sigma(\dot{s}_\alpha) \equiv_c \dot{s}_{\theta(\alpha)}$, we get
\[
p \Vdash \forall x \big(x \in \sigma(\dot{\mathcal{R}}) \text{ iff } x \in [\kappa]^{\le\omega} \text{ and } \forall \alpha \in \kappa \ (\dot{s}_{\theta(\alpha)} \le_c x \Rightarrow \theta(\alpha) \in \dot{f}(x))\big).
\]
Finally, as $\Vdash \forall x \in [\kappa]^{\le\omega} \ (\dot{f}(x) \equiv_c x)$, we conclude
\[
p \Vdash \forall x (x \in \sigma(\dot{\mathcal{R}}) \text{ iff } \dot{f}(x) \in \dot{\mathcal{R}}),
\]
which suffices to prove that \eqref{eq:100} holds. By the Symmetry Lemma and \eqref{eq:99}, it follows that
$$
p \Vdash (\sigma(\dot{\mathcal{R}}), \subseteq) \vDash \phi(\sigma(\dot{a})).
$$
By \eqref{eq:100},
$$
p \Vdash (\dot{\mathcal{R}}, \subseteq) \vDash \phi\big(\dot{f}(\sigma(\dot{a}))\big).
$$

By assumption $p \Vdash 0 \in \dot{a}$ and $1 \not\in \dot{a}$. Therefore, by the Symmetry Lemma, $p$ forces $0 \in \sigma(\dot{a})$ and $1 \not\in \sigma(\dot{a})$. Thus,  $p$ forces $1 = \theta(0) \in \dot{f}(\sigma(\dot{a}))$, and, in particular, it forces $\dot{a} \neq \dot{f}(\sigma(\dot{a}))$.

Going back to $\mathrm{L}[G]$, if we let $b$ be the evaluation $\dot{f}(\sigma(\dot{a}))$ according to the generic filter $G$, then $b \neq a$ and $(\mathcal{R}, \subseteq) \vDash \phi(b)$, as we wanted to show.
\end{proof}

\paragraph{\textbf{Non-real degrees.}} The situation becomes more complicated if we are interested in the constructibility degrees of $\mathrm{L}[G]$, without focusing solely on the real ones. For instance, consider the set $\mathbf{S} \coloneqq \{\mathbf{s}_n \mid n\in \omega\}$, i.e., the set of all the constructibility degrees of the generic Sacks reals. It can be shown, via an argument very similar to the one employed in Theorem~\ref{thm:main3}, that the set $\mathbf{S}$ is amorphous in $\mathrm{L}(\mathbf{S})$. Recall that a set is said to be \emph{amorphous} if it is infinite and it is not the disjoint union of two infinite subsets. This implies, in particular, that in $\mathrm{L}[G]$, the set $\mathbf{S}$ is not equiconstructible with any real. If it were, then $\mathrm{L}(\mathbf{S})$ would satisfy the axiom of choice, and thus there would be no amorphous sets in it. 

\section{Proof of Theorem~\ref{thm:main4}}\label{sec:last}

By Theorems~\ref{thm:main1}-\ref{thm:main3}, we cannot hope to improve Theorem~\ref{thm:representation} by devising, in $\mathrm{L}[G]$, an isomorphism between $(\mathcal{D}_c, \le_c)$ and $([\kappa]^{\le\omega}, \subseteq)$. This is because $([\kappa]^{\le\omega}, \subseteq)$ is a $\sigma$-complete lattice and it is far from being rigid, having $2^\kappa$-many automorphisms. In this section, we show that this fact is not accidental, as we prove in $\mathsf{ZF}$ that $(\mathcal{P}(\omega), \subseteq)$ cannot be isomorphic to any ideal of $(\mathcal{D}_c, \le_c)$.

\begin{lemma}\label{prop:neg}
Suppose that $f: (\mathcal{P}(\omega), \subseteq) \rightarrow (\mathcal{D}_c, \le_c)$ is an order-embedding, then
\[
\big\{f(\{n\}) \mid n \in \omega\big\} \le_c f(\omega) \Longrightarrow \mathbb{R} \subseteq \mathrm{L}[f(\omega)].
\]
\end{lemma}
\begin{proof}
Fix an order-embedding $f:(\mathcal{P}(\omega), \subseteq) \rightarrow (\mathcal{D}_c, \le_c)$ such that
\[
A \coloneqq \big\{f(\{n\}) \mid n\in \omega\big\} \le_c f(\omega).
\]
Note that $A$ is the image via $f$ of the set of atoms of $\mathcal{P}(\omega)$. Fix a real $x$. Since in $\mathrm{L}[f(\omega)]$ the axiom of choice holds, we can fix an injection $g: \omega \rightarrow A$ such that $g \le_c f(\omega)$. Let 
\[
y \coloneqq \big\{n \in \omega \mid \exists m \in x \ \big(f(\{n\}) = g(m)\big)\big\}.
\]
Since $f$ is an order-embedding, we must have $f(y) \le_c f(\omega)$. But then 
\[
x =  g^{-1}(\{\mathbf{d} \in A \mid \mathbf{d} \le_c f(y)\})\le_c f(\omega),
\]
where the equality follows from the fact that for any $n \in\omega$, $f(\{n\}) \le_c f(y)$ if and only if $n \in y$, since $f$ is an order-embedding; thus, $x$ is constructible relative to $f(\omega)$ because $A$, $g$ and $f(y)$ are constructible relative to $f(\omega)$, and because of the absoluteness of the constructibility preorder. Thus, every real is constructible relative to $f(\omega)$.
\end{proof}

\begin{proof}[Proof of Theorem~{\upshape\ref{thm:main4}}]
Suppose that $g:(\mathcal{P}(\omega), \subseteq) \rightarrow (\mathcal{D}_c, \le_c)$ is an order-embedding with $\mathrm{ran}(g)$ being an ideal of $(\mathcal{D}_c, \le_c)$, towards a contradiction. Consider the map $f:(\mathcal{P}(\omega), \subseteq) \rightarrow (\mathcal{D}_c, \le_c)$ defined as follows: for every $x \in \mathcal{P}(\omega)$, let $f(x) = g(\{n+1 \mid n \in x\})$. Clearly, $f$ is still an order-embedding, and its range is still an ideal of $\mathcal{D}_c$. Moreover, $g(\{0\}) \not\le_c g(\omega\setminus \{0\}) = f(\omega)$, since $g$ is assumed to be an embedding. Therefore, the real degree $g(\{0\})$ witnesses that not every real is constructible relative to $f(\omega)$. Furthermore, 
\begin{equation}\label{eq:finalsection}
\big\{f(\{n\}) \mid n \in \omega\big\} = \{\mathbf{d} \le_c f(\omega) \mid \mathbf{d}\text{ is an atom of } \mathcal{D}_c\} \le_c f(\omega),
\end{equation}
where the equality follows from $\mathrm{ran}(f)$ being an ideal of $\mathcal{D}_c$, and the second $\le_c$ follows from the absoluteness of the constructibility preorder. However, \eqref{eq:finalsection} contradicts Lemma~\ref{prop:neg}.
\end{proof}

We conclude with the following question, which asks to what extent Theorem~\ref{thm:main4} still holds if we weaken the assumption on the range of the isomorphism.

\begin{question}
Is it consistent with $\mathsf{ZFC}$ that $(\mathcal{P}(\omega), \subseteq)$ embeds into $(\mathcal{D}_c, \le_c)$ as a lattice? As a join-semilattice? As partial order?
\end{question}

\printbibliography
\end{document}